\documentclass[11pt,twoside,a4paper]{article}
\usepackage{amssymb}
\usepackage{latexsym}
\usepackage[latin1]{inputenc}
\usepackage[T1]{fontenc}   
\usepackage[english]{babel}
\input{xy}
\xyoption{all}

\newcommand{\tun}{\begin{picture}(5,0)(-2,-1)
\put(0,0){\circle*{2}}
\end{picture}}

\newcommand{\tdun}[1]{\begin{picture}(10,5)(-2,-1)
\put(0,0){\circle*{2}}
\put(3,-2){\tiny #1}
\end{picture}}

\newcommand{\arbreun}{
\begin{picture}(30,60)(-10,0)
\put(0,0){\line(0,0){10}}
\put(0,10){\line(1,1){20}}
\put(0,10){\line(-1,1){10}}
\put(10,20){\line(-1,1){10}}
\put(22,32){.}
\put(24,34){.}
\put(26,36){.}
\put(30,40){\line(1,1){10}}
\put(40,50){\line(1,1){10}}
\put(30,40){\line(-1,1){10}}
\put(40,50){\line(-1,1){10}}
\put(-17,23){\tiny $\sigma(1)$}
\put(-7,33){\tiny $\sigma(2)$}
\put(1,53){\tiny $\sigma(n-2)$}
\put(11,63){\tiny $\sigma(n-1)$}
\put(41,63){\tiny $\sigma(n)$}
\put(3,7){\tiny $a_1$}
\put(13,17){\tiny $a_2$}
\put(35,39){\tiny $a_{n-2}$}
\put(45,49){\tiny $a_{n-1}$}
\end{picture}}

\newcommand{\arbredeux}{
\begin{picture}(30,100)(-10,0)
\put(0,0){\line(0,0){10}}
\put(0,10){\line(1,1){20}}
\put(0,10){\line(-1,1){10}}
\put(10,20){\line(-1,1){10}}
\put(22,32){.}
\put(24,34){.}
\put(26,36){.}
\put(30,40){\line(1,1){10}}
\put(40,50){\line(1,1){10}}
\put(30,40){\line(-1,1){10}}
\put(40,50){\line(-1,1){10}}
\put(-17,23){\tiny $\sigma(1)$}
\put(-7,33){\tiny $\sigma(2)$}
\put(1,53){\tiny $\sigma(i-1)$}
\put(14,63){\tiny $\sigma(i)$}
\put(3,7){\tiny $\bas$}
\put(13,17){\tiny $\bas$}
\put(35,39){\tiny $\bas$}
\put(45,49){\tiny $m$}
\put(52,62){.}
\put(54,64){.}
\put(56,66){.}
\put(60,70){\line(1,1){10}}
\put(70,80){\line(1,1){10}}
\put(60,70){\line(-1,1){10}}
\put(70,80){\line(-1,1){10}}
\put(31,83){\tiny $\sigma(n-2)$} 
\put(41,93){\tiny $\sigma(n-1)$}
\put(71,93){\tiny $\sigma(n)$}
\put(65,69){\tiny $m$}
\put(75,79){\tiny $m$}
\end{picture}}

\newcommand{\arbretrois}{
\begin{picture}(30,100)(-10,0)
\put(0,0){\line(0,0){10}}
\put(0,10){\line(1,1){20}}
\put(0,10){\line(-1,1){10}}
\put(10,20){\line(-1,1){10}}
\put(22,32){.}
\put(24,34){.}
\put(26,36){.}
\put(30,40){\line(1,1){10}}
\put(40,50){\line(1,1){10}}
\put(30,40){\line(-1,1){10}}
\put(40,50){\line(-1,1){10}}
\put(-12,23){\tiny $1$}
\put(-2,33){\tiny $2$}
\put(6,53){\tiny $i-1$}
\put(26,63){\tiny $i$}
\put(3,7){\tiny $\bas$}
\put(13,17){\tiny $\bas$}
\put(35,39){\tiny $\bas$}
\put(45,49){\tiny $m$}
\put(52,62){.}
\put(54,64){.}
\put(56,66){.}
\put(60,70){\line(1,1){10}}
\put(70,80){\line(1,1){10}}
\put(60,70){\line(-1,1){10}}
\put(70,80){\line(-1,1){10}}
\put(36,83){\tiny $n-2$} 
\put(48,93){\tiny $n-1$}
\put(76,93){\tiny $n$}
\put(65,69){\tiny $m$}
\put(75,79){\tiny $m$}
\end{picture}}

\newcommand{\arbrequatre}{
\begin{picture}(30,100)(-10,0)
\put(0,0){\line(0,0){10}}
\put(0,10){\line(1,1){20}}
\put(0,10){\line(-1,1){10}}
\put(10,20){\line(-1,1){10}}
\put(22,32){.}
\put(24,34){.}
\put(26,36){.}
\put(30,40){\line(1,1){10}}
\put(40,50){\line(1,1){10}}
\put(30,40){\line(-1,1){10}}
\put(40,50){\line(-1,1){10}}
\put(-17,23){\tiny $\sigma(1)$}
\put(-7,33){\tiny $\sigma(2)$}
\put(1,53){\tiny $\sigma(i-1)$}
\put(14,63){\tiny $\sigma(i)$}
\put(3,7){\tiny $\haut$}
\put(13,17){\tiny $\haut$}
\put(35,39){\tiny $\haut$}
\put(45,49){\tiny $m$}
\put(52,62){.}
\put(54,64){.}
\put(56,66){.}
\put(60,70){\line(1,1){10}}
\put(70,80){\line(1,1){10}}
\put(60,70){\line(-1,1){10}}
\put(70,80){\line(-1,1){10}}
\put(31,83){\tiny $\sigma(n-2)$} 
\put(41,93){\tiny $\sigma(n-1)$}
\put(71,93){\tiny $\sigma(n)$}
\put(65,69){\tiny $m$}
\put(75,79){\tiny $m$}
\end{picture}}

\newcommand{\bdtroisun}[2]{
\begin{picture}(30,40)(-20,0)
\put(0,0){\line(0,0){10}}
\put(0,10){\line(1,1){10}}
\put(0,10){\line(-1,1){20}}
\put(-10,20){\line(1,1){10}}
\put(-9,7){\tiny #1}
\put(-19,17){\tiny #2}
\end{picture}}
\newcommand{\bdtroisdeux}[2]{
\begin{picture}(30,40)(-10,0)
\put(0,0){\line(0,0){10}}
\put(0,10){\line(1,1){20}}
\put(0,10){\line(-1,1){10}}
\put(10,20){\line(-1,1){10}}
\put(3,7){\tiny #1}
\put(13,17){\tiny #2}
\end{picture}}

\newcommand{\bddtroisun}[5]{
\begin{picture}(30,40)(-20,0)
\put(0,0){\line(0,0){10}}
\put(0,10){\line(1,1){10}}
\put(0,10){\line(-1,1){20}}
\put(-10,20){\line(1,1){10}}
\put(-9,7){\tiny #1}
\put(-19,17){\tiny #2}
\put(-21,32){\tiny #3}
\put(-1,32){\tiny #4}
\put(9,22){\tiny #5}
\end{picture}}
\newcommand{\bddtroisdeux}[5]{
\begin{picture}(30,40)(-10,0)
\put(0,0){\line(0,0){10}}
\put(0,10){\line(1,1){20}}
\put(0,10){\line(-1,1){10}}
\put(10,20){\line(-1,1){10}}
\put(3,7){\tiny #1}
\put(13,17){\tiny #2}
\put(-11,22){\tiny #3}
\put(-1,32){\tiny #4}
\put(19,32){\tiny #5}
\end{picture}}

\title{The operads of planar forests are Koszul}
\date{}
\author{L. Foissy \\ \\
{\small{\it Laboratoire de Mathématiques, Université de Reims}}\\
\small{{\it Moulin de la Housse - BP 1039 - 51687 REIMS Cedex 2, France}}\\
\small{e-mail : loic.foissy@univ-reims.fr}}

\newcommand{\D}{{\cal D}}
\newcommand{\F}{\mathbf{F}}
\newcommand{\T}{\mathbf{T}}
\renewcommand{\P}{\mathbb{P}}
\newcommand{\haut}{\nearrow }
\newcommand{\bas}{\searrow }
\newcommand{\hauttimes}{{\hspace{1mm} \otimes \hspace{-4.5mm} \haut\hspace{1mm}}}
\newcommand{\G}{\P_\bas^!}
\renewcommand{\H}{\P_\haut^!}

\newtheorem{defi}{\indent Definition}

\newtheorem{theo}[defi]{\indent Theorem}
\newtheorem{lemma}[defi]{\indent Lemma}
\newtheorem{cor}[defi]{\indent Corollary}

\newenvironment{proof}{{\bf Proof.}}{\hfill $\Box$}

\begin{document}

\maketitle

ABSTRACT. We describe the Koszul dual of two quadratic operads on planar forests introduced to study the infinitesimal Hopf algebra
of planar rooted trees and prove that these operads are Koszul.\\


AMS CLASSIFICATION. 05C05, 18D50.

\tableofcontents

\section*{Introduction}

The Hopf algebra of planar rooted trees, described in \cite{Foissy1,Holtkamp}, is a non-commutative version of the Hopf algebra
of rooted tree introduced in \cite{Connes,Kreimer1,Kreimer2,Kreimer3} in the context of Quantum Field Theories and Renormalisation.
An infinitesimal version of this object is introduced in \cite{Foissy2}, and is related to  two operads on planar forests in \cite{Foissy3}.
These two operads, denoted by $\P_\bas$ and $\P_\haut$, are presented in the following way:
\begin{enumerate}
\item $\P_\bas$ is generated by $m$ and $\bas \in \P_\bas(2)$, with relations:
$$ \left\{ \begin{array}{rcl}
m\circ (\bas,I)&=&\bas \circ (I,m),\\
m\circ (m,I)&=&m\circ (I,m),\\
\bas \circ (m,I)&=&\bas \circ (I,\bas).
\end{array}\right.$$
\item $\P_\haut$ is generated by $m$ and $\haut \in \P_\haut(2)$, with relations:
$$ \left\{ \begin{array}{rcl}
m\circ (\haut,I)&=&\haut \circ (I,m),\\
m\circ (m,I)&=&m\circ (I,m),\\
\haut \circ (\haut,I)&=&\haut \circ (I,\haut).
\end{array}\right.$$
\end{enumerate}
The algebra of planar rooted trees is both the free $\P_\haut$- and $\P_\bas$-algebra generated by $\tun$,
with products $\haut$ and $\bas$ given by certain graftings.

The operads $\P_\haut$ and $\P_\bas$  are quadratic. Our aim in this note is to prove that they are both Koszul, in the sense of \cite{Ginzburg}.
We describe their Koszul dual (it turns out that they are quotient of $\P_\bas$ and $\P_\haut$)
and the associated homology of $\P_\haut$- or $\P_\bas$-algebras. We compute these homologies for free objects
and prove that they are concentrated in degree $0$. This proves that these operads are Koszul.

\section{Operads of planar forests}

\subsection{Presentation}

We work in this text with operads, whereas we worked in \cite{Foissy3} with non-$\Sigma$-operads.
In other terms, we replace the non-$\Sigma$-operads of \cite{Foissy3} by their symmetrization \cite{Markl}.

\begin{defi}\textnormal{\begin{enumerate}
\item $\P_\bas$ is generated, as an operad, by $m$ and $\bas$, with the relations:
$$ \left\{ \begin{array}{rcl}
\bas \circ (m, I)&=&\bas \circ (I,\bas),\\
\bas \circ (I,m)&=&m \circ (\bas,I),\\
m \circ (m, I)&=&m \circ (I,m).
\end{array}\right.$$
\item $\P_\haut$ is generated, as an operad, by $m$ and $\haut$, with the relations:
$$ \left\{ \begin{array}{rcl}
\haut \circ (\haut, I)&=&\haut \circ (I,\haut),\\
\haut \circ (I,m)&=&m \circ (\haut,I),\\
m \circ (m, I)&=&m \circ (I,m).
\end{array}\right.$$
\end{enumerate}}\end{defi}

{\bf Remarks.} \begin{enumerate}
\item Graphically, the relations defining $\P_\bas$ can be written in the following way:
$$\bddtroisun{$\bas $}{$m$}{$1$}{$2$}{$3$}=\bddtroisdeux{$\bas $}{$\bas $}{$1$}{$2$}{$3$},\hspace{1cm}
 \bddtroisun{$m$}{$m$}{$1$}{$2$}{$3$}=\bddtroisdeux{$m$}{$m$}{$1$}{$2$}{$3$},\hspace{1cm}
 \bddtroisun{$m$}{$\bas $}{$1$}{$2$}{$3$}=\bddtroisdeux{$\bas $}{$m$}{$1$}{$2$}{$3$}.$$
\item We denote by $\tilde{\P}_\bas$ the sub-non-$\Sigma$-operad of $\P_\bas$ generated by $m$ and $\bas$.
Then $\P_\bas$ is the symmetrization of $\tilde{\P}_\bas$.
\item Graphically, the relations of $\H$ can be written in the following way:
$$\bdtroisun{$\haut $}{$\haut $}=\bdtroisdeux{$\haut $}{$\haut $},\hspace{1cm}
 \bdtroisun{$m$}{$m$}=\bdtroisdeux{$m$}{$m$},\hspace{1cm}
 \bdtroisun{$m$}{$\haut $}=\bdtroisdeux{$\haut $}{$m$}.$$
\item  We denote by $\tilde{\P}_\haut$ the sub-non-$\Sigma$-operad of $\P_\haut$ generated by $m$ and $\haut$.
Then $\P_\haut$ is the symmetrization of $\tilde{\P}_\haut$.
\end{enumerate}
Both of these non-$\Sigma$-operads admits a description in terms of planar forests \cite{Foissy3}.
In particular, the dimension of $\tilde{\P}_\bas(n)$ and $\tilde{\P}_\haut(n)$ is given by the $n$-th Catalan number \cite{Stanley,Stanley2}.
Multiplying by a factorial, for all $n \geq 1$:
$$\dim\P_\bas(n)=\dim\P_\haut(n)=\frac{(2n)!}{(n+1)!}.$$
In particular, $\dim \P_\bas(2)=\dim \P_\haut(2)=4$ and $\dim \P_\bas(3)=\dim \P_\haut(3)=30$.

\subsection{Free algebras on these operads}

We described in \cite{Foissy3} the free $\P_\bas$- and $\P_\haut$-algebras on one generators, using planar rooted trees.
We here generalise (without proof) these results. Let $\D$ be any set. We denote by $\T^\D$ the set of planar trees decorated by $\D$ 
and by $\F^\D$ the set of non-empty planar forests decorated by $\D$. 
\begin{enumerate}
\item The free $\P_\bas$-algebra generated by $\D$ has the set $\F^\D$ as a basis. The product $m$ is given by concatenation of forests. 
For all $F$, $G \in \F^\D$, the product $F \bas G$ is obtained by grafting $F$ on the root of $G$, on the left.
\item The free $\P_\haut$-algebra generated by $\D$ has the set $\F^\D$ as a basis. The product $m$ is given by concatenation of forests. 
For all $F$, $G \in \F^\D$, the product $F \haut G$ is obtained by grafting $F$ on the left leaf of $G$.
\end{enumerate}
In both cases, we identified $d \in \D$ with $\tdun{$d$} \in \F^\D$. Moreover, for all $F \in \F^\D$, $F\bas \tdun{$d$}=F \haut \tdun{$d$}$ is the tree
obtained by grafting the trees of $F$ on a common root decorated by $d$: this tree will be denoted by $B_d(F)$.

\section{The operad $\P_\bas$ is Koszul}

\subsection{Koszul dual of $\P_\bas$}

(See \cite{Ginzburg,Markl} for the notion of Koszul duality for quadratic operads). We denote by $\G$ the Koszul dual of $\P_\bas$.

\begin{theo} \label{2}
The operad $\G$ is generated by $m$ and $\bas \in \G(2)$, with the relations:
$$ \left\{ \begin{array}{rcl}
\bas\circ(m,I)&=&\bas\circ(I,\bas),\\
m\circ(m,I)&=&m\circ(I,m),\\
m\circ(\bas,I)&=&\bas\circ(I,m),\\
\bas\circ(\bas,I)&=&0,\\
m\circ(I,\bas)&=&0.
\end{array}\right.$$
\end{theo}

\begin{proof} Let $\P(E)$ be the operad freely generated by the $S_2$-module freely generated by $m$ and $\bas$. 
Then $\P_\bas$ can be written $\P_\bas=\P(E)/(R)$, where $R$ is a sub-$S_3$-module of $\P(E)(3)$. As $\dim(\P(E))=48$ and
$\dim(\P_\bas(3))=30$, $\dim(R)=18$. So $\dim(R^\perp)=48-18=30$. We then verify that the given relations for $\G$ are indeed in $R^\perp$,
that each of them generates a free $S_3$-module, which are in direct sum. So these relations generate entirely $\P(E)(3)$. \end{proof}\\

{\bf Remarks.} \begin {enumerate}
\item So $\G$ is a quotient of $\P_\bas$.
\item Moreover, $\G$ is the symmetrisation of the non-$\Sigma$-operad $\tilde{\P}_\bas^!$ generated by $m$ and $\bas$ and the relations:
$$ \left\{ \begin{array}{rcl}
\bas\circ(m,I)&=&\bas\circ(I,\bas),\\
m\circ(m,I)&=&m\circ(I,m),\\
m\circ(\bas,I)&=&\bas\circ(I,m),\\
\bas\circ(\bas,I)&=&0,\\
m\circ(I,\bas)&=&0.
\end{array}\right.$$
This is a general fact: the Koszul dual of the symmetrisation of a quadratic non-$\Sigma$ operad is itself the symmetrisation 
of a certain quadratic non-$\Sigma$-operad.
\item Graphically, the relations defining $\G$ can be written in the following way:
$$\bddtroisun{$\bas $}{$m$}{$1$}{$2$}{$3$}=\bddtroisdeux{$\bas $}{$\bas $}{$1$}{$2$}{$3$},\hspace{1cm}
\bddtroisun{$m$}{$m$}{$1$}{$2$}{$3$}=\bddtroisdeux{$m$}{$m$}{$1$}{$2$}{$3$},\hspace{1cm}
\bddtroisun{$m$}{$\bas $}{$1$}{$2$}{$3$}=\bddtroisdeux{$\bas $}{$m$}{$1$}{$2$}{$3$},$$
$$\bddtroisun{$\bas $}{$\bas $}{$1$}{$2$}{$3$}=0,\hspace{1cm}\bddtroisdeux{$m$}{$\bas $}{$1$}{$2$}{$3$}=0.$$ 
\end{enumerate}

\subsection{Free $\G$-algebras}

Let $V$ be finite-dimensional vector space. We put:
$$\left\{ \begin{array}{rcl}
T_\bas(V)(n)&=&\displaystyle \bigoplus_{k=1}^n V^{\otimes n} \mbox{ for all }n \geq 1,\\
T_\bas(V)&=&\displaystyle \bigoplus_{n=1}^{\infty} T_\bas(V)(n).
\end{array}\right.$$
In order to distinguish the different copies of $V^{\otimes n}$, we put:
$$T(V)(n)=\bigoplus_{k=1}^n \underbrace{\left(A\otimes \ldots \otimes  A\otimes \dot{A}\otimes A \otimes \ldots \otimes A\right)}_{
\mbox{the $k$-th copy of $A$ is pointed.}}.$$ 
The elements of $A\otimes \ldots \otimes  A\otimes \dot{A}\otimes A \otimes \ldots \otimes A$ will be denoted by
$v_1\otimes \ldots \otimes v_{k-1} \otimes \dot{v}_k \otimes v_{k+1}\otimes \ldots \otimes v_n.$
We define $m$ and $\bas$ over $T_\bas(V)$ in the following way: for $v=v_1\otimes \ldots \otimes \dot{v}_k \otimes \ldots \otimes v_m$ 
and $w=w_1\otimes \ldots \otimes \dot{w}_l \otimes \ldots \otimes w_n$,
\begin{eqnarray*}
vw&=& \left\{ \begin{array}{l}
0\mbox{ if } l\neq 1, \\
v_1\otimes \ldots \otimes \dot{v}_k \otimes \ldots \otimes v_m \otimes
w_1\otimes \ldots \otimes w_n \mbox{ if } l=1;\\
\end{array} \right.\\ \\
v \bas w&=& \left\{ \begin{array}{l}
0\mbox{ if } k\neq 1, \\
v_1\otimes \ldots \otimes v_m \otimes
w_1\otimes \ldots \otimes \dot{w}_l \otimes \ldots \otimes w_n\mbox{ if } k=1.
\end{array}\right.
\end{eqnarray*}

\begin{lemma}
$T_\bas(V)$ is a $\G$-algebra generated by $V$. 
\end{lemma}

\begin{proof} Let us first show that the relations of the $\G$-algebras are satisfied. 
Let $u=u_1 \otimes \ldots \otimes \dot{u}_j \otimes \ldots \otimes u_m$, $v=v_1\otimes \ldots \otimes \dot{v}_k \otimes \ldots \otimes v_n$ and 
$w=w_1\otimes \ldots \otimes \dot{w}_l \otimes \ldots \otimes w_p$.
\begin{eqnarray*}
(uv)\bas w&=&0 \mbox{ if $j \neq 1$ or $k \neq 1$},\\
&=&u_1\otimes \ldots \otimes u_m \otimes v_1 \otimes \ldots \otimes v_n \otimes w_1 \otimes 
\ldots \otimes \dot{w}_l \otimes \ldots \otimes w_p \mbox{ if $j=k=1$},\\
u \bas (v \bas w)&=&0 \mbox{ if $j \neq 1$ or $k \neq 1$},\\
&=&u_1\otimes \ldots \otimes u_m \otimes v_1 \otimes \ldots \otimes v_n \otimes w_1 \otimes 
\ldots \otimes \dot{w}_l \otimes \ldots \otimes w_p \mbox{ if $j=k=1$},\\ \\
(uv)w&=&0\mbox{ if $k\neq 1$ or $l\neq 1$},\\
&=&u_1 \otimes \ldots \otimes \dot{u}_j \ldots \otimes u_m \otimes 
v_1 \otimes \ldots \otimes v_n \otimes w_1 \otimes \ldots \otimes w_p \mbox{ if $k=l=1$},\\
u(vw)&=&0\mbox{ if $k\neq 1$ or $l\neq 1$},\\
&=&u_1 \otimes \ldots \otimes \dot{u}_j \ldots \otimes u_m \otimes 
v_1 \otimes \ldots \otimes v_n \otimes w_1 \otimes \ldots \otimes w_p \mbox{ if $k=l=1$},\\ \\
(u \bas v)w&=&0\mbox{ if $j\neq 1$ or $l\neq 1$},\\
&=&u_1 \otimes \ldots \otimes u_m \otimes v_1 \otimes \ldots \otimes \dot{v}_k \ldots \otimes v_n \otimes 
w_1 \otimes \ldots \otimes w_p \mbox{ if $j=l=1$},\\
u \bas (vw)&=&0\mbox{ if $j\neq 1$ or $l\neq 1$},\\
&=&u_1 \otimes \ldots \otimes u_m \otimes v_1 \otimes \ldots \otimes \dot{v}_k \ldots \otimes v_n \otimes 
w_1 \otimes \ldots \otimes w_p \mbox{ if $j=l=1$},\\ \\
(u \bas v) \bas w&=&0,\\ \\
u(v \bas w)&=&0.
\end{eqnarray*}
So $(T_\bas(V),m,\bas)$ is a $\G$-algebra. Moreover, for all $v_1,\ldots,v_n \in V$:
$$v_1\otimes \ldots \otimes \dot{v}_k \otimes \ldots \otimes v_n=(\dot{v}_1\ldots \dot{v}_{k-1})\bas (\dot{v}_k \ldots \dot{v}_n).$$
Hence, $T_\bas(V)$ is generated by $V$. \end{proof}\\

The $\G$-algebra $T_\bas(V)$ is also graded by putting $V$ in degree $1$. 
It is then a quotient of the free $\G$-algebra generated by $V$, which is:
$$\bigoplus_{n=0}^{\infty} \tilde{\P}_\bas^!(n) \otimes V^{\otimes n}.$$
So, for all $n \in \mathbb{N}$, $\dim(\tilde{\P}_\bas^!(n) \otimes V^{\otimes n})\geq \dim(T_\bas(V)(n))$,
so $\dim(\tilde{\P}_\bas^!(n))\geq n$ and $\dim(\G(n))\geq n n!$.

\begin{lemma}
For all $n \in \mathbb{N}$, $dim(\G(n))\leq n n!$.
\end{lemma}

\begin{proof} $\G(n)$ is linearly generated by the binary trees with $n$ indexed leaves, whose internal vertices are decorated by $m$ and $\bas$.
 By the four first relations of $\G$, we obtain that $\G(n)$ is generated by the trees of the following form: 
$$\arbreun,$$
with $\sigma \in S_n$, $a_1,\ldots a_{n-1} \in \{m,\bas\}$. With the last relation, we deduce that $\G(n)$ is generated by the trees 
of the following form:
$$\arbredeux,$$
where $\sigma \in S_n$, $1 \leq i \leq n$. Hence, $dim(\G(n))\leq nn!$. \end{proof} \\

As a consequence:

\begin{theo}
Let $n \geq 1$.
\begin{enumerate}
\item $dim(\G(n))=nn!$.
\item $\G(n)$ is freely generated, as a $S_n$-module, by the following trees:
$$\arbretrois,$$
where $1 \leq i \leq n$.
\item $T_\bas(V)$ is the free $\G$-algebra generated by $V$.
\end{enumerate}
\end{theo}

\subsection{Homology of a $\P_\bas$-algebra}

Let us now describe the cofree $\P_\bas$-algebra cogenerated by $V$. By duality, it is equal to $T_\bas(V)$ as a vector space, 
with coproducts given in the following way: for $v=v_1 \otimes \ldots \otimes \dot{v}_k \otimes \ldots \otimes v_m$,
\begin{eqnarray*}
\Delta(v)&=&\sum_{i=k}^{m-1} (v_1 \otimes \ldots \otimes \dot{v}_k \otimes \ldots \otimes v_i) \otimes (\dot{v}_{i+1}\otimes \ldots \otimes v_m),\\
\Delta_\bas(v)&=&\sum_{i=1}^{k-1} (\dot{v}_1\otimes \ldots \otimes v_i)\otimes (v_{i+1}\otimes \ldots \otimes \dot{v}_k \otimes \ldots \otimes v_m).
\end{eqnarray*}

Let $A$ be a $\P_\bas$-algebra. The homology complex of $A$ is given by the shifted cofree coalgebra $T_\bas(V)[-1]$,
with differential $d:T_\bas(V)(n) \longrightarrow T_\bas(V)(n-1)$, uniquely determined by the following conditions:
\begin{enumerate}
\item for all $a,b \in A$, $d(\dot{a} \otimes b)=ab$.
\item for all $a,b \in A$, $d(a\otimes \dot{b})=a\bas b$.
\item Let $\theta: T_\bas(A) \longrightarrow T_\bas(A)$ be the following application:
$$ \theta: \left\{\begin{array}{rcl}
T_\bas(A) & \longrightarrow &T_\bas(A)\\
x& \longrightarrow &(-1)^{degree(x)}x \mbox{ for all homogeneous $x$}.
\end{array}\right. $$
Then $d$ is a $\theta$-coderivation: for all $x\in T_\bas(A)$,
\begin{eqnarray*}
\Delta(d(x))&=&(d\otimes Id +\theta \otimes Id) \circ \Delta(x),\\
\Delta_\bas(d(x))&=&(d\otimes Id +\theta \otimes Id) \circ \Delta_\bas(x).
\end{eqnarray*} \end{enumerate}
So, $d$ is the application which sends the element $v_1 \otimes \ldots \otimes \dot{v}_k \otimes \ldots \otimes v_m$ to:
\begin{eqnarray*}
&&\sum_{i=1}^{k-2} (-1)^{i-1} v_1 \otimes \ldots \otimes v_i v_{i+1} 
\otimes \ldots \otimes \dot{v}_k \otimes \ldots \otimes v_m\\
&&+(-1)^{k-2} v_1 \otimes \ldots \otimes \overbrace{v_{k-1} \bas v_k}^{.} \otimes \ldots \otimes v_m\\
&&+(-1)^{k-1} v_1 \otimes \ldots \otimes \overbrace{v_k v_{k+1}}^{.} \otimes \ldots \otimes v_m\\
&&+\sum_{i=k+1}^{n-1} (-1)^{i-1} v_1 \otimes \ldots \otimes \dot{v}_k \otimes \ldots \otimes v_i v_{i+1} \otimes \ldots \otimes v_m.
\end{eqnarray*}
The homology of this complex will be denoted by $H_*^\bas(A)$. More clearly, for all $n \in \mathbb{N}$:
$$H_n^\bas(A)=\frac{Ker\left(d_{\mid T_\bas(A)(n+1)}\right)}{Im\left(d_{\mid T_\bas(A)(n+2)}\right)}.$$

{\bf Examples.} Let $v_1,v_2,v_3 \in A$.
$$\left\{ \begin{array}{rcl}
d(v_1)&=&0,\\
d(\dot{v_1}\otimes v_2)&=&v_1v_2,\\
d(v_1\otimes \dot{v_2})&=&v_1\bas v_2,\\
d(\dot{v_1}\otimes v_2\otimes v_3)&=&\overbrace{v_1v_2}^{.}\otimes v_3- \dot{v_1}\otimes v_2v_3,\\
d(v_1\otimes \dot{v_2}\otimes v_3)&=&\overbrace{v_1\bas v_2}^{.}\otimes v_3- v_1\otimes \overbrace{v_2v_3}^{.},\\
d(v_1\otimes v_2\otimes \dot{v_3})&=&v_1v_2\otimes \dot{v_3}- v_1\otimes \overbrace{v_2\bas v_3}^{.}.
\end{array}\right.$$
So:
$$\left\{ \begin{array}{rcl}
d^2(\dot{v_1}\otimes v_2\otimes v_3)&=&(v_1v_2)v_3- v_1(v_2v_3),\\
d^2(v_1\otimes \dot{v_2}\otimes v_3)&=&(v_1\bas v_2) v_3- v_1\bas (v_2v_3),\\
d^2(v_1\otimes v_2\otimes \dot{v_3})&=&(v_1v_2)\bas v_3- v_1\bas (v_2\bas v_3).
\end{array}\right.$$
So the nullity of $d^2$ on $T_\bas(A)(3)$ is equivalent to the three relations defining $\P_\bas$-algebras (this is a general fact \cite{Ginzburg}).
 In particular:
$$H_0^\bas(A)=\frac{A}{A.A+A\bas A}.$$

\subsection{Homology of free $\P_\bas$-algebras}

The aim of this paragraph is to prove the following result:

\begin{theo}
let $N \geq 1$ and let $A$ be the free $\P_\bas$-algebra generated by $D$ elements. 
Then $H_0^\bas(A)$ is $D$-dimensional; if $n \geq 1$, $H_n^\bas(A)=(0)$.
\end{theo}

\begin{proof} {\it Preliminaries.} We put, for all $k,n \in \mathbb{N}^*$:
$$\left\{ \begin{array}{rcl}
C_n&=&T_\bas(A)(n),\\[2mm]
C_n^k&=& \underbrace{A\otimes \ldots \otimes  \dot{A} \otimes \ldots \otimes A}_{\mbox{$A$ in position $k$}} \subseteq C_n \mbox{ if } k\leq n,\\[8mm]
C_n^{\leq k}&=&\bigoplus_{i\leq k,n} C_n^i\subseteq C_n.
\end{array}\right.$$
For all $k \in \mathbb{N}^*$, $C_*^{\leq k}$ is a subcomplex of $C_*$. In particular, $C_*^{\leq 1}$ is isomorphic to the complex defined by
$C'_n=A^{\otimes n}$, with a differential defined by:
$$d' : \left\{\begin{array}{rcl}
A^{\otimes n}& \longrightarrow & A^{\otimes (n-1)}\\
v_1\otimes \ldots \otimes v_n & \longrightarrow &
\displaystyle \sum_{i=1}^{n-1} (-1)^{i-1} v_1\otimes \ldots \otimes v_{i-1} \otimes v_iv_{i+1}
\otimes v_{i+2}\otimes \ldots \otimes v_n.
\end{array}\right. $$
The homology of $C'_*$ is then the shifted Hochschild homology of $A$. As $A$ is a free (non unitary) associative algebra, 
this homology is concentrated in degree $1$. So, $Ker\left(d_{| C_n^{\leq 1}}\right)\subseteq Im(d)$ if $n \geq 2$.\\

{\it First step.} Let us fix $n \geq 2$ and let us show that $Ker\left(d_{| C_n^{\leq k}}\right)\subseteq Im(d)$ for all $1 \leq k \leq n-1$ 
by induction on $k$. For $k=1$, this is already done. Let us assume that $2\leq k<n$ and $Ker\left(d_{| C_n^{\leq k-1}}\right)\subseteq Im(d)$.
Let $x=\displaystyle \sum_{i=1}^{k} x_i\in Ker\left(d_{| C_n^{\leq k}}\right)$, with $x_i \in C_n^i$. If $x_k=0$, then 
$x \in Ker\left(d_{| C_n^{\leq k-1}}\right)$ by the induction hypothesis. Otherwise, we put:
$$x_k=\sum v_1\otimes \ldots \otimes \dot{v_k} \otimes \ldots \otimes v_n.$$
We project $d(x)$ over $C_{n-1}^k$. we obtain:
\begin{eqnarray*}
0&=&\sum_{i=1}^{k-1} \pi_k(d(x_i))+\sum\sum_{i=1}^{k-2} (-1)^{i-1} \pi_k(v_1 \otimes \ldots \otimes v_{i-1} \otimes v_iv_{i+1}
\otimes v_{i+2} \otimes \ldots \otimes \dot{v}_k \otimes \ldots \otimes v_n)\\
&&+\sum (-1)^{k-2} \pi_k( v_1\otimes \ldots \otimes v_{k-2} \otimes \overbrace{v_{k-1}\bas v_k}^{.}\otimes v_{k+1}\otimes \ldots \otimes v_n)\\
&&+\sum (-1)^{k-1} \pi_k(v_1\otimes \ldots  \otimes v_{k-1}\otimes \overbrace{v_k v_{k+1}}^{.}\otimes v_{k+2}\otimes \ldots \otimes v_n)\\
&&+\sum \sum_{i=k+1}^{n-1}(-1)^{i-1} \pi_k(v_1\otimes \ldots \otimes \dot{v}_k \otimes \ldots 
\otimes v_{i-1}\otimes v_iv_{i+1}\otimes v_{i+2}\otimes \ldots \otimes v_n)\\
&=&0+0+0+ \sum (-1)^{k-1} v_1\otimes \ldots  \otimes \otimes v_{k-1}
\otimes \overbrace{v_k v_{k+1}}^{.}\otimes v_{k+2}\otimes \ldots \otimes v_n\\
&&+\sum\sum_{i=k+1}^{n-1}(-1)^{i-1}  v_1\otimes \ldots \otimes \dot{v}_k \otimes \ldots 
\otimes v_{i-1}\otimes v_iv_{i+1}\otimes v_{i+2}\otimes \ldots \otimes v_n\\
&=&(-1)^{k-1} \sum v_1\otimes \ldots \otimes v_{k-1} \otimes d'(v_k\otimes \ldots \otimes v_n).
\end{eqnarray*}
Hence, we can suppose that $d'(v_k\otimes \ldots \otimes v_n)=0$. As $n-k+1\geq 2$ and the complex $C'_*$ is exact in degree $n-k+1\geq 2$,
there exists $\sum w_k \otimes \ldots \otimes w_{n+1} \in A^{\otimes (n-k+2)}$, such that:
$$d'\left(\sum w_k \otimes \ldots \otimes w_{n+1} \right)=v_k\otimes \ldots \otimes v_n.$$
We put $w=\sum v_1\otimes \ldots \otimes v_{k-1} \otimes \left(\sum\dot{w_k} \otimes \ldots \otimes w_{n+1}\right)$.
Then $d(w)=x_k+C_n^{k-1}$, so $x-d(w)\in C_n^{k-1}$. As $Im(d)\subseteq Ker(d)$, $x-d(w) \in Ker\left(d_{| C_n^{\leq k-1}}\right)\subseteq Im(d)$ 
by the induction hypothesis. So, $x \in Im(d)$.\\

{\it Second step.} Let us show that $Ker\left(d_{|C_n^{\leq n}}\right)\subseteq Im(d)$ if $n \geq 3$. Let $x \in Ker\left(d_{| C_n^{\leq n}}\right)$.
As before, we put $x=\displaystyle \sum_{i=1}^n x_i$, with $x_i \in C_n^i$ and:
$$x_n=\sum_i v^i_1\otimes \ldots \otimes \dot{v^i_k} \otimes \ldots \otimes v^i_n.$$
We can assume that the $v^i_j$'s are homogeneous. Let us fix an integer $N$, greater than the degree of $x_n$ and an integer $M$,
smaller than $\displaystyle\max_i\{weight(v^i_n)\}$. Let us show by decreasing induction the following property:
For all$x  \in Ker\left(d_{| C_n^{\leq n}}\right)$ of weight $\leq N$ and such that $\displaystyle\max_i\{weight(v^i_n)\}\geq M$, then $x \in Im(d)$.
If $M>N$, such an $x$ is zero and the result is trivial. Let us assume the property at rank $M+1$ and let us prove it at rank $M$.
Let $A_M$ be the homogeneous component (for the weight of forests) of degree $M$ of $A$. We project $d(x)$ over 
$A\otimes \ldots \otimes \dot{A_M}$. Then: 
$$0=\varpi_M(d(x))=\sum_{i,\: weight(v_n^i)=M} d'(v_1^i\otimes\ldots \otimes v_{n-1}^i) \otimes \dot{v_n^i}.$$
Hence, we can suppose that, for all $i$ such that $weight(v_n^i)=M$, $d'(v_1^i\otimes\ldots \otimes v_{n-1}^i)=0$. As $n \geq 3$ 
and $C'_*$ is exact at $n-1\geq 2$, there exists $\displaystyle \sum_j w^{i,j}_1 \otimes \ldots \otimes w^{i,j}_n \in A^{\otimes n}$ such that:
$$d'\left(\sum_j w^{i,j}_1 \otimes \ldots \otimes w^{i,j}_n \right)=v_1^i\otimes\ldots \otimes v_{n-1}^i.$$
As $d'$  is homogeneous for the weight, the weight of this element can be supposed smallest than the weight of 
$v_1^i\otimes\ldots \otimes v_{n-1}^i$. We then put $\displaystyle w=\sum_{i,\: poids(v_n^i)=M}  \left( \sum_j w^{i,j}_1 \otimes \ldots \otimes
w^{i,j}_n\right)\otimes \dot{v^i_n}$. So $x-d(w)$ is in $x \in Ker\left(d_{| C_n^{\leq n}}\right)$, with weight $\leq N$, 
and satisfies the property on the $v_n^i$'s for $M+1$. By induction hypothesis, $x-d(w)\in Im(d)$, so $x \in Im(d)$.\\

Hence, if  $n\geq 3$, $Ker\left(d_{| C_n^{\leq n}}\right) \subseteq Im(d)$. As $C_n^{\leq n}=C_n$, for all $n\geq 3$,
$d(C_{n+1})\subseteq Ker\left(d_{| C_n}\right) \subseteq d(C_{n+1})$. Consequently, if $n \geq 2$, $H_n^{\bas}(A)=(0)$.\\

{\it Third step.} We now compute $H_1^\bas(A)$. We take an element $x \in C_2$ and show that it belongs to $Im(d)$.
This $x$ can be written under the form:
$$x= \sum_{F,G\in \F^\D-\{1\}} a_{F,G} F \otimes \dot{G}-\sum_{F,G\in \F^\D-\{1\}} b_{F,G} \dot{F} \otimes G.$$
So:
$$d(x)= \sum_{F,G\in \F^\D-\{1\}} a_{F,G} F \bas G-\sum_{F,G\in \F-\{1\}} b_{F,G} FG.$$
Hence, the following assertions are equivalent:
\begin{enumerate}
\item $d(x)=0$.
\item For all $H\in \F^\D-\{1\}$, $\displaystyle  \sum_{F\bas G=H}a_{F,G} =\sum_{FG=H}b_{F,G}$.
\end{enumerate}

{\it First case.} For all $F,G\in \F^\D-\{1\}$, $a_{F,G}=0$, that is to say $x \in \dot{A}\otimes A$.
So $d(x)=d'(x')$. As $C'_*$ is exact in degree $2$, there exists $v_1\otimes v_2 \otimes v_3 \in A^{\otimes 3}$ such that
$\displaystyle d'(v_1\otimes v_2 \otimes v_3)=\sum_{F,G} b_{F,G} F\otimes G$.
Consequently, $\displaystyle d(\dot{v_1}\otimes v_2 \otimes v_3)=\sum_{F,G} b_{F,G} \dot{F}\otimes G=x$.\\

{\it Second case.} $x=F_1\otimes \dot{F_2}-\dot{G_1}\otimes G_2$, $F_1,F_2,G_1,G_2 \in \F^\D$, such that $F_1\bas F_2=G_1G_2=H$.
We put $H=t_1\ldots t_n$ and $t_1=B_d(s_1\ldots s_m)$, $t_1,\ldots,t_n,s_1,\ldots,s_m\in \T^\D$. There exists $i\in \{1,\ldots,n-1\}$
such that $G_1=t_1\ldots t_i$ and $G_2=t_{i+1}\ldots t_n$; there exists $j\in \{1,\ldots,m-1\}$ such that $F_1=s_1\ldots s_j$ and
$F_2=B_d(s_{j+1}\ldots s_m)t_2\ldots t_n$. Then:
\begin{eqnarray*}
&&d(s_1\ldots s_j \otimes \overbrace{B_d(s_{j+1}\ldots s_m)t_2\ldots t_i }^{.} \otimes t_{i+1} \ldots t_n)\\
&=&\overbrace{(s_1\ldots s_j) \bas B_d(s_{j+1}\ldots s_m)t_2\ldots t_i }^{.} \otimes t_{i+1} \ldots t_n\\
&&-s_1\ldots s_j \otimes \overbrace{B_d(s_{j+1}\ldots s_m)t_2\ldots t_i t_{i+1} \ldots t_n}^{.}\\
&=&\dot{G_1}\otimes G_2 -F_1 \otimes \dot{F_2}.
\end{eqnarray*}
So, $x \in Im(d)$.\\

{\it Third case.} We suppose now the following condition: 
$$(a_{F,G}\neq 0 )\:\Longrightarrow \: (G\notin \T^\D).$$ 
So, $x$ can be written:
$$x= \sum_{F,G\in \F^\D,\: t\in \T^\D} a_{F,tG} F\otimes \overbrace{tG}^{.}-\sum_{F,G\in \F^\D} b_{F,G} \dot{F}\otimes G.$$
By the second case, $F\otimes \overbrace{tG}^{.} - \overbrace{F\bas t}^{.}\otimes G \in Im(d)\subseteq Ker(d)$.
So the following element belongs to $Ker(d)$:
\begin{eqnarray*}
&&x-\sum_{F,G\in \F^\D,\: t\in \T^\D} a_{F,tG} (F\otimes \overbrace{tG}^{.} - \overbrace{F\bas t}^{.}\otimes G)\\
&=&- \sum_{F,G\in \F^\D} b_{F,G} \dot{F}\otimes G +\sum_{F,G\in \F^\D,\: t\in \T^\D} a_{F,tG}  \overbrace{F\bas t}^{.}\otimes G.
\end{eqnarray*}
By the first case, this element belongs to $Im(d)$, so $x \in Im(d)$.\\

{\it Fourth case.} We suppose now the following condition:
$$(a_{F,G}\neq 0 )\:\Longrightarrow \: (G\notin \T^\D \mbox{ ou } G=\tdun{$d$},\: d\in \D).$$
Let $H=B_d^+(t_1\ldots t_n) \in \T^\D$, different from a single root. Then:
$$0=\sum_{F\bas G=H}a_{F,G}-\sum_{FG=H}b_{F,G}=\sum_{i=1}^n a_{t_1\ldots t_i,B_d(t_{i+1}\ldots t_n)}-0
=a_{t_1\ldots t_n,\tdun{$d$}}+0=a_{F,\tdun{$d$}}.$$
Consequently, for all $F\in \F^\D$, $d\in \D$, $a_{F,\tdun{$d$}}=0$. By the third case, $x \in Im(d)$.\\

{\it General case.} The following element belongs to $Ker(d)$:
\begin{eqnarray*}
x'&=& x+\sum_{F,G \in \F^\D,\: d\in \D} a_{F,B_d(G)}d(F\otimes G \otimes \dot{\tdun{$d$}})\\
&=&x+\sum_{F,G \in \F^\D,\: d\in \D} a_{F,B_d(G)}FG \otimes \dot{\tdun{$d$}}
-\sum_{F,G \in \F^\D,\: d\in \D} a_{F,B_d(G)}F\otimes \overbrace{G\bas \tdun{$d$}}^{.}\\
&=&x+ \sum_{F,G \in \F^\D,\: d\in \D} a_{F,B_d(G)}FG \otimes \dot{\tdun{$d$}}-\sum_{F,G \in \F^\D,\: d\in \D} a_{F,B_d(G)}F\otimes \dot{B_d(G)}\\
&=& \sum_{F\in \F^\D,\:G \in \F^\D-\T^\D}a_{F,G} F\otimes \dot{G}+\sum_{F\in \F^\D,\:d\in \D}a_{F,G} F\otimes \dot{\tdun{$d$}}\\
&&-\sum_{F,G\in \F^\D} b_{F,G} \dot{F}\otimes G+\sum_{F,G \in \F^\D,\: d\in \D} a_{F,B_d(G)}FG \otimes \dot{\tdun{$d$}}.
\end{eqnarray*}
So $x'$ satisfies the condition of the fourth case, so $x'\in Im(d)$. Hence, $x\in Im(d)$.
This proves finally that $Ker(d_{|C_2})=d(C_3)$, so $H^\bas_1(A)=(0)$\\

It remains to compute $H_0^\bas(A)$. This is equal to $A/(A.A+A\bas A)$, so a basis of $H_0^\bas(A)$ is given by the trees of weight $1$, 
so $dim(H_0^\bas(A))=D$. \end{proof}\\

As an immediate corollary:

\begin{cor}
The operad $\P_\bas$ is Koszul.
\end{cor}

\section{The operad $\P_\haut$ is Koszul}

\subsection{Koszul dual of $\P_\haut$}

We denote by $\H$ the Koszul dual of $\P_\haut$.

\begin{theo}
The operad $\H$ is generated by $m$ and $\haut \in \H(2)$, with the relations:
$$ \left\{ \begin{array}{rcl}
\haut\circ(\haut,I)&=&\haut\circ(I,\haut),\\
m\circ(m,I)&=&m\circ(I,m),\\
m\circ(\haut,I)&=&\haut\circ(I,m),\\
\haut\circ(m,I)&=&0,\\
m\circ(I,\haut)&=&0.
\end{array}\right.$$
\end{theo}

\begin{proof} Similar as the proof of theorem \ref{2}. \end{proof}\\

{\bf Remarks.} \begin {enumerate}
\item So $\H$ is a quotient of $\P_\haut$.
\item The operad $\H$ is the symmetrization of the non-$\Sigma$-operad $\tilde{\P}_\haut^!$, generated by $m$ and $\haut$, with relations:
$$ \left\{ \begin{array}{rcl}
\haut\circ(\haut,I)&=&\haut\circ(I,\haut),\\
m\circ(m,I)&=&m\circ(I,m),\\
m\circ(\haut,I)&=&\haut\circ(I,m),\\
\haut\circ(m,I)&=&0,\\
m\circ(I,\haut)&=&0.
\end{array} \right.$$
\item Graphically, the relations of $\H$ can be written in the following way:
$$\bdtroisun{$\haut $}{$\haut $}=\bdtroisdeux{$\haut $}{$\haut $},\hspace{1cm} \bdtroisun{$m$}{$m$}=\bdtroisdeux{$m$}{$m$},\hspace{1cm}
\bdtroisun{$m$}{$\haut $}=\bdtroisdeux{$\haut $}{$m$},$$
$$\bdtroisun{$\haut $}{$m$}=0,\hspace{1cm} \bdtroisdeux{$m$}{$\haut $}=0.$$ 
\end{enumerate}

\subsection{Free $\H$-algebras}

Let $V$ be finite-dimensional vector space. We put:
$$\left\{ \begin{array}{rcl}
T_\haut(V)(n)&=&\displaystyle \bigoplus_{k=1}^n V^{\otimes n} \mbox{ for all }n\geq 1,\\
T_\haut(V)&=&\displaystyle \bigoplus_{n=1}^{\infty} T_\haut(V)(n).
\end{array}\right.$$
In order to distinguish the different copies of $V^{\otimes n}$, we put:
$$T(V)(n)=\bigoplus_{k=1}^n \left(\underbrace{A\hauttimes \ldots \hauttimes  A\hauttimes A\otimes A \otimes \ldots \otimes A}_{
\mbox{$(k-1)$ signs $\hauttimes $}} \right).$$
The elements of $A\hauttimes \ldots \hauttimes  A\otimes \ldots \otimes A$ will be denoted by $v_1\hauttimes \ldots \hauttimes v_k \otimes
\ldots\otimes v_n$. We define $m$ and $\haut$ over $T_\haut(V)$ in the following way: 
for $v=v_1\hauttimes \ldots \hauttimes v_k \otimes \ldots \otimes v_m$ and  $w=w_1\hauttimes \ldots \hauttimes w_l \otimes \ldots \otimes w_n$,
\begin{eqnarray*}
vw&=& \left\{ \begin{array}{l}
0\mbox{ if } l\neq 1, \\
v_1\hauttimes \ldots \hauttimes v_k \otimes \ldots \otimes v_m \otimes w_1\otimes \ldots \otimes w_n \mbox{ if } l=1;\\
\end{array} \right.\\ \\
v \haut w&=& \left\{ \begin{array}{l}
0\mbox{ if } k\neq m-1, \\
v_1\hauttimes \ldots \hauttimes v_m \hauttimes w_1\hauttimes \ldots \hauttimes w_l \otimes \ldots \otimes w_n\mbox{ if } k=1.
\end{array}\right. \end{eqnarray*}

As for $\P_\bas$, we can prove the following result:

\begin{theo}
Let $n \geq 1$.
\begin{enumerate}
\item $dim(\H(n))=nn!$.
\item $\H(n)$ is freely generated, as a $S_n$-module, by the following trees:
$$\arbrequatre,$$
where $1 \leq i \leq n$.
\item $T_\haut(V)$ is the free $\H$-algebra generated by $V$.
\end{enumerate} \end{theo}

\subsection{Homology of a $\P_\haut$-algebra}

Let us now describe the cofree $\P_\haut$-algebra cogenerated by $V$. By duality, it is equal to $T_\haut(V)$ as a vector space, 
with coproducts given in the following way: for $v=v_1 \hauttimes \ldots \hauttimes v_k \otimes \ldots \otimes v_m$,
\begin{eqnarray*}
\Delta(v)&=&\sum_{i=k}^{m-1} (v_1 \hauttimes \ldots \hauttimes v_k \otimes \ldots \otimes v_i) \otimes (v_{i+1}\otimes \ldots \otimes v_m),\\
\Delta_\haut(v)&=&\sum_{i=1}^{k-1} (v_1\hauttimes \ldots \hauttimes v_i)\otimes (v_{i+1}\hauttimes \ldots \hauttimes v_k \otimes \ldots \otimes v_m).
\end{eqnarray*}

Let $A$ be a $\P_\haut$-algebra. The homology complex of $A$ is given by the shifted cofree coalgebra $T_\haut(V)[-1]$,
with differential $d:T_\haut(V)(n) \longrightarrow T_\haut(V)(n-1)$, uniquely determined by the following conditions:
\begin{enumerate}
\item for all $a,b \in A$, $d(a \otimes b)=ab$.
\item for all $a,b \in A$, $d(a\hauttimes b)=a\haut b$.
\item Let $\theta: T_\haut(A) \longrightarrow T_\haut(A)$ be the following application:
$$ \theta: \left\{\begin{array}{rcl}
T_\haut(A) & \longrightarrow &T_\haut(A)\\
x& \longrightarrow &(-1)^{degree(x)}x \mbox{ for all homogeneous $x$}.
\end{array}\right. $$
Then $d$ is a $\theta$-coderivation: for all $x\in T_\haut(A)$,
\begin{eqnarray*}
\Delta(d(x))&=&(d\otimes Id +\theta \otimes Id) \circ \Delta(x),\\
\Delta_\haut(d(x))&=&(d\otimes Id +\theta \otimes Id) \circ \Delta_\haut(x).
\end{eqnarray*} \end{enumerate}
So, $d$ is the application which sends the element $v_1 \otimes \ldots \otimes \dot{v}_k \otimes \ldots \otimes v_m$ to:
\begin{eqnarray*}
&&d(v_1\hauttimes \ldots \hauttimes v_k \otimes \ldots \otimes v_n)\\
&=&\sum_{i=1}^{k-1} (-1)^{i-1} v_1 \hauttimes \ldots \hauttimes v_{i-1} \hauttimes v_i\haut v_{i+1}
\hauttimes v_{i+2} \hauttimes \ldots \hauttimes v_k \otimes \ldots \otimes v_n\\
&&+\sum_{i=k}^{n-1}(-1)^{i-1}v_1\hauttimes \ldots \hauttimes v_k \otimes \ldots \otimes 
v_{i-1} \otimes v_i v_{i+1}\otimes v_{i+2} \otimes \ldots  \otimes v_n.
\end{eqnarray*}
This homology will be denoted by $H_*^\haut(A)$. More clearly, for all $n \in \mathbb{N}$:
$$H_n^\haut(A)=\frac{Ker\left(d_{\mid T_\haut(A)(n+1)}\right)}{Im\left(d_{\mid T_\haut(A)(n+2)}\right)}.$$

{\bf Examples.} Let $v_1,v_2,v_3 \in A$.
$$\left\{ \begin{array}{rcl}
d(v_1)&=&0,\\
d(v_1\otimes v_2)&=&v_1v_2,\\
d(v_1\hauttimes v_2)&=&v_1\haut v_2,\\
d(v_1\otimes v_2\otimes v_3)&=&v_1v_2\otimes v_3- v_1\otimes v_2v_3,\\
d(v_1\hauttimes v_2\otimes v_3)&=&v_1\haut v_2\otimes v_3- v_1\hauttimes v_2v_3,\\
d(v_1\hauttimes v_2\hauttimes v_3)&=&v_1\haut v_2\hauttimes v_3- v_1\hauttimes v_2\haut v_3.
\end{array}\right.$$
So:
$$\left\{ \begin{array}{rcl}
d^2(v_1\otimes v_2\otimes v_3)&=&(v_1v_2) v_3- v_1(v_2v_3),\\
d^2(v_1\hauttimes v_2\otimes v_3)&=&(v_1\haut v_2) v_3- v_1\haut(v_2v_3),\\
d^2(v_1\hauttimes v_2\hauttimes v_3)&=&(v_1\haut v_2)\haut v_3- v_1\haut(v_2\haut v_3).
\end{array}\right.$$
So the nullity of $d^2$ on $T_\haut(A)(3)$ is equivalent to the three relations defining $\P_\haut$-algebras, as for $\P_\bas$. In particular:
$$H_0^\haut(A)=\frac{A}{A.A+A\haut A}.$$

\subsection{Homology of free $\P_\haut$-algebras}

The aim of this paragraph is to prove the following result:

\begin{theo}
let $N \geq 1$ and let $A$ be the free $\P_\haut$-algebra generated by $D$ elements.
Then $H_0^\haut(A)$ is $D$-dimensional; if $n \geq 1$, $H_n^\haut(A)=(0)$.
\end{theo}

\begin{proof} {\it Preliminaries.} We put, for $k,n \in \mathbb{N}^*$:
$$ \left\{ \begin{array}{rcl}
{C'}_n&=&T_\haut(A)(n),\\[2mm]
{C'}_n^k&=& \underbrace{A \hauttimes \ldots \hauttimes  A \otimes \ldots \otimes A}_{
\mbox{$k-1$ signs $\hauttimes$}} \subseteq {C'}_n \mbox{ if } k\leq n, \\[8mm]
{C'}_n^{\leq k}&=&\bigoplus_{i\leq k,n} {C'}_n^i\subseteq {C'}_n.
\end{array}\right.$$
For all $k \in \mathbb{N}^*$, ${C'}_*^{\leq k}$ is a subcomplex of ${C'}_n$. In particular, ${C'}_*^{\leq 1}$ is isomorphic to the complex 
defined by ${C'}_n=A^{\otimes n}$, with differential given by:
$$d': \left\{\begin{array}{rcl}
A^{\otimes n}& \longrightarrow & A^{\otimes (n-1)}\\
a_1\otimes \ldots \otimes a_n & \longrightarrow & \displaystyle \sum_{i=1}^{n-1} (-1)^{i-1} a_1\otimes \ldots \otimes a_{i-1} \otimes a_ia_{i+1}
\otimes a_{i+2}\otimes \ldots \otimes a_n.
\end{array}\right. $$
Hence, the homology of ${C'}_*$ is the (shifted) Hochschild homology of $A$.
As $A$ is a free (non unitary) associative algebra, this homology is concentrated in degree $1$. So:
\begin{equation} \label{E1}
Ker\left(d_{| {C'}_n^{\leq 1}}\right)\subseteq Im(d) \mbox{ if  }n \geq 2.
\end{equation}

Moreover, $C'_*$ admits a subcomplex defined by $C''_*(n)=A\hauttimes \ldots \hauttimes A$, with differential given by:
$$d : \left\{\begin{array}{rcl}
C''_*(n)& \longrightarrow & C''_*(n-1)\\
v_1\hauttimes \ldots \hauttimes v_n & \longrightarrow &\displaystyle \sum_{i=1}^{n-1} (-1)^{i-1} v_1\hauttimes \ldots \hauttimes v_{i-1} \otimes
v_i\haut v_{i+1} \hauttimes v_{i+2}\hauttimes \ldots \hauttimes v_n.
\end{array}\right. $$
Hence, the homology of this subcomplex is the shifted Hochschild homology of the associative algebra $(A,\haut)$.

\begin{lemma} \label{11}
Every forest $F \in \F^\D-\{1\}$ can be uniquely written as $F_1\haut \ldots \haut F_n$, 
where the $F_i$'s are elements of $\F^\D$ of the form  $F_i=\tun_{d_i}G_i$.
\end{lemma}

\begin{proof} {\it Existence.} By induction on the weight of $F$. If $weight(F)=1$, $F=\tun_d$ and the result is obvious.
If $weight(F) \geq 2$, we put $F=B^+_d(H_1)H_2$, with $weight(H_1)<weight(F)$. If $H_1=1$, the result is obvious. 
If $H_1 \neq 1$, we apply the induction hypothesis on $H_1$, so it can be written as $H_1=F_1\haut \ldots \haut F_n$, with $F_i=\tun_{d_i}G_i$. 
We put  $F_{n+1}=\tun_d H_2$, so $F=F_1\haut \ldots \haut F_{n+1}$.\\

{\it Unicity.} By induction on the weight of $F$. If $weight(F)=1$, then $F=\tun_d$ and this is obvious. If $weight(F) \geq 2$, 
we put $F=B_d(H_1)H_2$, with $weight(H_1)<weight(F)$. If $F=F_1\haut\ldots \haut F_n$, then $F_n=\tun_d H_2$ 
and $F_1\haut \ldots \haut F_{n-1}=H_1$. Hence, $F_n$ is unique. We conclude with the induction hypothesis. \end{proof} \\

This lemma implies that $(A,\haut)$ is freely generated by forests of the form $\tun_d G$. So:
\begin{equation} \label{E2}
Ker\left(d_{| C''_n}\right)\subseteq Im(d) \mbox{ if  }n \geq 2.
\end{equation}

{\it First step.} Let us fix $n \geq 2$. We show by induction on $k$ the following property:
$$Ker\left(d_{| {C'}_n^{\leq k}}\right)\subseteq Im(d) \mbox{ for all } 1\leq k \leq n-1.$$
For $k=1$, this is (\ref{E1}). Let us suppose $2\leq k<n$ and $Ker\left(d_{| {C'}_n^{\leq k-1}}\right)\subseteq Im(d)$. 
Let $x=\displaystyle \sum_{i=1}^{k} x_i\in Ker\left(d_{| {C'}_n^{\leq k}}\right)$, with $x_i \in {C'}_n^i$. If $x_k=0$, then 
$x \in Ker\left(d_{| {C'}_n^{\leq k-1}}\right)$ and the induction hypothesis holds. We then suppose $x_k \neq 0$, and we put:
$$x_k=\sum v_1\hauttimes \ldots \hauttimes v_k \otimes \ldots \otimes v_n.$$
Let us project $d(x)$ over ${C'}_{n-1}^k$. We get:
\begin{eqnarray*}
&&\sum_{i=1}^{k-1} \pi_k(d(x_i))+\sum_{i=1}^{k-1}(-1)^{i-1} 
\sum \pi_k(v_1\hauttimes \ldots \hauttimes v_i\haut v_{i+1} \hauttimes \ldots \hauttimes v_k \otimes \ldots \otimes v_n)\\
&&+\sum_{i=k}^{n-1} (-1)^{i-1} \sum \pi_k(v_1\hauttimes \ldots \hauttimes v_k \otimes \ldots \otimes v_iv_{i+1} \otimes \ldots \otimes v_n)\\
&=&0+0+(-1)^{k-1} \sum v_1\hauttimes \ldots \hauttimes v_{k-1}\hauttimes d'(v_k\otimes \ldots \otimes v_n)\\
&=&0.
\end{eqnarray*}
Hence, we can suppose $d'(v_k\otimes \ldots \otimes v_n)=0$. As $n-k+1\geq 2$, by (\ref{E1}), there exists an element
$\sum w_k \otimes \ldots \otimes w_{n+1} \in A^{\otimes (n-k+2)}$, such that:
$$d'\left(\sum w_k \otimes \ldots \otimes w_{n+1} \right)=v_k\otimes \ldots \otimes v_n.$$
We put $w=\sum v_1\hauttimes \ldots \hauttimes v_{k-1} \hauttimes \left(\sum w_k  \otimes \ldots \otimes w_{n+1}\right)$.
Then, $d(w)=x_k+{C'}_n^{k-1}$, so $x-d(w)\in {C'}_n^{k-1}$. As $Im(d)\subseteq Ker(d)$, 
$x-d(w) \in Ker\left(d_{| {C'}_n^{\leq k-1}}\right)\subseteq Im(d)$ by the induction hypothesis. Hence, $x \in Im(d)$.\\

{\it Second step.} Let us show that, if $n \geq 3$,  $Ker\left(d_{|{C'}_n^{\leq n}}\right)\subseteq Im(d)$. Take $x \in Ker\left(d_{| C_n^{\leq n}}\right)$,
written as $x=\displaystyle \sum_{i=1}^n x_i$, with $x_i \in C_n^i$ and $\displaystyle x_n=\sum_i v^i_1\hauttimes \ldots \hauttimes v^i_n$.
We can suppose the $v^i_j$'s homogeneous. Let us fix an integer $N$, greater than the degree of $x_n$, and an integer $M$, 
smaller than $\displaystyle\min_i\{weight(v^i_n)\}$. Let us show by a decreasing induction on $M$ the following property:
for all $x  \in Ker\left(d_{| {C'}_n^{\leq n}}\right)$, of weight $\leq N$, and such that $\displaystyle\min_i\{weight(v^i_n)\}\geq M$, then $x \in Im(d)$.
If $M>N$, such an $x$ is zero, and the result is obvious. Suppose the result at rank $M+1$ and let us show it at rank $M$.
Let $A_M$ be the homogeneous (for the weight) component of degree $M$ of $A$ and let us project $d(x)$ over 
$A\hauttimes \ldots \hauttimes A\hauttimes A_M$. We get: 
$$0=\varpi_M(d(x))=\sum_{i,\: weight(v_n^i)=M} d(v_1^i\hauttimes\ldots \hauttimes v_{n-1}^i)\hauttimes v_n^i.$$
Hence, we can suppose that, for all $i$ such that $weight(v_n^i)=M$, $d(v_1^i\hauttimes\ldots \hauttimes v_{n-1}^i)=0$. As $n \geq 3$, 
by (\ref{E2}), there exists $\displaystyle \sum_j w^{i,j}_1 \hauttimes \ldots \hauttimes w^{i,j}_n \in {C'}_n$ such that:
$$d\left(\sum_j w^{i,j}_1 \hauttimes \ldots \hauttimes w^{i,j}_n \right)=v_1^i\hauttimes\ldots \hauttimes v_{n-1}^i.$$
As $d$  is homogeneous for the weight, we can suppose that the weight of this element is smaller than the weight of 
$v_1^i\otimes\ldots \otimes v_{n-1}^i$. We then put:
$$w=\sum_{i,\: weight(v_n^i)=M} \sum_j w^{i,j}_1 \hauttimes \ldots \hauttimes w^{i,j}_n\hauttimes v_n^i.$$
So $x-d(w) \in Ker\left(d_{| {C'}_n^{\leq n}}\right)$, with a weight $\leq N$, and satisfies the property on the $v_n^i$'s for $M+1$. 
By the induction hypothesis, $x-d(w)\in Im(d)$, so $x \in Im(d)$.\\

So, if $n\geq 2$, as ${C'}_n^{\leq n}=C'_n$, $H_n^{\haut}(A)=(0)$.\\

{\it Third step.} We now compute $H_1^\haut(A)$. We take an element $x\in C'_2$ and show that it belongs to $Im(d)$. 
This element can be written as:
$$x= \sum_{F,G\in \F^\D-\{1\}} a_{F,G} F \hauttimes G-\sum_{F,G\in \F^\D-\{1\}} b_{F,G} F \otimes G.$$
so:
$$d(x)= \sum_{F,G\in \F^\D-\{1\}} a_{F,G} F \haut G-\sum_{F,G\in \F^\D-\{1\}} b_{F,G} FG.$$
As a consequence, the following assertions are equivalent:
\begin{enumerate}
\item $d(x)=0$.
\item for all $H\in \F^\D-\{1\}$, $\displaystyle \sum_{F\haut G=H}a_{F,G} =\sum_{FG=H}b_{F,G}$.
\end{enumerate}

{\it First case.} For all $F,G\in \F^\D-\{1\}$, $a_{F,G}=0$, that is to say $x \in A\otimes A$: then the result comes directly from (\ref{E1}).\\

{\it Second case.}  $x=F_1\hauttimes F_2-G_1\otimes G_2$, $F_1,F_2,G_1,G_2 \in \F^\D$, such that $F_1\haut F_2=G_1G_2=H$.
We put $H=t_1\ldots t_n$ et $t_1=H_1\haut \ldots \haut H_m$, $t_1,\ldots,t_n\in \T^\D$, the  $H_i$'s of the form $\tun_{d_i}H'_i$ 
(lemma \ref{11}). Then there exists $i\in \{1,\ldots,n-1\}$, such that $G_1=t_1\ldots t_i$ and $G_2=t_{i+1}\ldots t_n$;
there exists $j\in \{1,\ldots,m-1\}$, such that $F_1=H_1\haut \ldots \haut H_j$ and $F_2=(H_{j+1}\haut \ldots \haut H_m)t_2\ldots t_n$. So:
\begin{eqnarray*}
&&d( H_1\haut \ldots\haut H_j \hauttimes (H_{j+1} \haut \ldots\haut H_m)t_2\ldots t_i  \otimes t_{i+1} \ldots t_n)\\
&=&(H_1 \haut \ldots \haut H_j)\haut (H_{j+1} \haut \ldots\haut H_m)t_2\ldots t_i  \otimes t_{i+1} \ldots t_n)\\
&&- H_1\haut \ldots\haut H_j \hauttimes  (H_{j+1} \haut \ldots\haut H_m)t_2\ldots t_i t_{i+1} \ldots t_n)\\
&=&G_1\otimes G_2-F_1\hauttimes F_2.
\end{eqnarray*}
Hence, $x \in Im(d)$.\\

{\it Third case.} We suppose that the following condition holds:
$$(a_{F,G}\neq 0 )\:\Longrightarrow \: (G\notin \T^\D).$$
So, $x$ can be written as:
$$x= \sum_{F,G\in \F^\D,\: t\in \T^\D} a_{F,tG} F\hauttimes tG -\sum_{F,G\in \F^\D} b_{F,G} F\otimes G.$$
By the second case, $F\hauttimes tG - F\haut t\otimes G \in Im(d)\subseteq Ker(d)$.
So, the following element belongs to $Ker(d)$:
\begin{eqnarray*}
&&x-\sum_{F,G\in \F^\D,\: t\in \T^\D} a_{F,tG} (F\hauttimes tG - F\haut t\otimes G)\\
&=&- \sum_{F,G\in \F^\D} b_{F,G} F\otimes G+\sum_{F,G\in \F^\D,\: t\in \T^\D} a_{F,tG}  F\haut t\otimes G.
\end{eqnarray*}
By the first case, this element belongs to $Im(d)$, so $x \in Im(d)$. \\

{\it Fourth case.} We suppose that the following condition holds:
$$(a_{F,G}\neq 0 )\:\Longrightarrow \: (G\notin \T^\D \mbox{ or } G=\tdun{$d$},\: d\in \D).$$
Let $H\in \F^\D-\{1\}$. Let us write $B^+_d(H)=H_1\haut \ldots \haut H_n$, $with H_i=\tun_{d_i} H'_i$ for all $i$ (lemma \ref{11}).
As $B^+_d(H)\in \T^\D$, $H_n=\tun_{d_n}$ and $H_1\haut\ldots \haut H_ {n-1}=H$. So:
\begin{eqnarray*}
0&=&\sum_{F\haut G=B_d(H)}a_{F,G}-\sum_{FG=B_d(H)}b_{F,G}\\
&=&\sum_{i=1}^n a_{H_1\haut \ldots\haut  H_i,H_{i+1}\haut \ldots \haut H_n}-0\\
&=&a_{H_1\haut \ldots \haut H_{n-1},\tdun{$d$}}+0\\
&=&a_{H,\tdun{$d$}}.
\end{eqnarray*}
(We used the condition on $x$ for the third equality). So, for all $F\in \F^\D$, $d\in \D$, we obtain $a_{F,\tdun{$d$}}=0$.
As a consequence, by the third case, $x \in Im(d)$.\\

{\it General case.} The following element belongs to $Ker(d)$:
\begin{eqnarray*}
x'&=& x+\sum_{F,G \in \F^\D,\: d\in \D} a_{F,B_d(G)}d(F\hauttimes G \hauttimes \tdun{$d$})\\
&=&x+\sum_{F,G \in \F^\D,\: d\in \D} a_{F,B_d(G)}F\haut G \hauttimes \tdun{$d$}
-\sum_{F,G \in \F^\D,\: d\in \D} a_{F,B_d(G)}F\hauttimes G\haut \tdun{$d$}\\
&=&x+\sum_{F,G \in \F^\D,\: d\in \D} a_{F,B_d(G)}F\haut G \hauttimes \tdun{$d$}-\sum_{F,G \in \F^\D,\: d\in \D} a_{F,B_d(G)}F\hauttimes B_d(G)\\
&=& \sum_{F\in \F^\D,\:G \in \F^\D-\T^\D}a_{F,G} F\hauttimes G+\sum_{F\in \F^\D,\:d\in \D}a_{F,G} F\hauttimes \tdun{$d$}\\
&&-\sum_{F,G\in \F^\D} b_{F,G} F\otimes G + \sum_{F,G \in \F^\D,\: d\in \D} a_{F,B_d(G)}FG \hauttimes \tdun{$d$}.
\end{eqnarray*}
So, $x'$ satisfies the condition of the fourth cas, so $x'\in Im(d)$. Hence, $x\in Im(d)$.\\

It remains to compute $H_0^\haut(A)$. This is equal to $A/(A.A+A\haut A)$, so a basis of $H_0^\haut(A)$ is given by the trees of weight $1$, 
so $dim(H_0^\haut(A))=D$. \end{proof}\\

As an immediate corollary:

\begin{cor}
The operad $\P_\haut$ is Koszul.
\end{cor}

\bibliographystyle{amsplain}
\bibliography{biblio}

\providecommand{\bysame}{\leavevmode\hbox to3em{\hrulefill}\thinspace}
\providecommand{\MR}{\relax\ifhmode\unskip\space\fi MR }
\providecommand{\MRhref}[2]{%
  \href{http://www.ams.org/mathscinet-getitem?mr=#1}{#2}
}
\providecommand{\href}[2]{#2}
\begin{thebibliography}{10}

\bibitem{Connes}
Alain Connes and Dirk Kreimer, \emph{Hopf algebras, {R}enormalization and
  {N}oncommutative geometry}, Comm. Math. Phys \textbf{199} (1998), no.~1,
  203--242, arXiv: hep-th/98 08042.

\bibitem{Foissy1}
Lo{\"\i}c Foissy, \emph{Les alg\`ebres de {H}opf des arbres enracin\'es, {I}},
  Bull. Sci. Math. \textbf{126} (2002), 193--239.

\bibitem{Foissy3}
Loïc Foissy, \emph{The infinitesimal {H}opf algebra and the operads of planar
  forests}, arXiv: 0901.2202.

\bibitem{Foissy2}
\bysame, \emph{The infinitesimal {H}opf algebra and the poset of planar
  forests}, arXiv: 0802.0442.

\bibitem{Ginzburg}
Victor Ginzburg and Mikhail Kapranov, \emph{Koszul duality for operads}, Duke
  Math. J. \textbf{76} (1994), no.~1, 203--272.

\bibitem{Holtkamp}
Ralf Holtkamp, \emph{Comparison of {H}opf {A}lgebras on {T}rees}, Arch. Math.
  (Basel) \textbf{80} (2003), no.~4, 368--383.

\bibitem{Kreimer1}
Dirk Kreimer, \emph{On the {H}opf algebra structure of pertubative quantum
  field theories}, Adv. Theor. Math. Phys. \textbf{2} (1998), no.~2, 303--334,
  arXiv: q-alg/97 07029.

\bibitem{Kreimer2}
\bysame, \emph{On {O}verlapping {D}ivergences}, Comm. Math. Phys. \textbf{204}
  (1999), no.~3, 669--689, arXiv: hep-th/98 10022.

\bibitem{Kreimer3}
\bysame, \emph{Combinatorics of (pertubative) {Q}uantum {F}ield {T}heory},
  Phys. Rep. \textbf{4--6} (2002), 387--424, arXiv: hep-th/00 10059.

\bibitem{Markl}
Martin Markl, Steve Shnider, and Jim Stasheff, \emph{Operads in algebra,
  topology and physics}, Mathematical Surveys and Monographs, no.~90, American
  Mathematical Society, Providence, RI, 2002.

\bibitem{Stanley}
Richard~P. Stanley, \emph{Enumerative combinatorics. {V}ol. 1.}, Cambridge
  Studies in Advanced Mathematics, no.~49, Cambridge University Press,
  Cambridge, 1997.

\bibitem{Stanley2}
\bysame, \emph{Enumerative combinatorics. {V}ol. 2.}, Cambridge Studies in
  Advanced Mathematics, no.~62, Cambridge University Press, Cambridge, 1999.

\end{thebibliography}

\end{document}